\documentclass{article}

\usepackage{amsmath,amsthm,amssymb,amsfonts}
\usepackage{verbatim}
\usepackage{array}

\usepackage{hyperref}

\usepackage{tikz}
\usepackage{tkz-berge, tkz-graph}
\tikzset{vertex/.style={circle,draw,fill,inner sep=0pt,minimum size=1mm}}

\usepackage{stmaryrd} 
\usepackage{hyperref}

\usepackage[colorinlistoftodos]{todonotes}

\usepackage{tikz}
\usepackage{tkz-berge, tkz-graph}

\tikzset{vertex/.style={circle,draw,fill,inner sep=0pt,minimum size=1mm}}

\theoremstyle{plain}
\newtheorem{thm}{Theorem}
\newtheorem{lem}[thm]{Lemma}
\newtheorem{prop}[thm]{Proposition}

\newtheorem{cor}[thm]{Corollary}

\theoremstyle{definition}

\newtheorem{definition}[thm]{Definition}
\newtheorem{exl}[thm]{Example}

\newtheorem{remark}[thm]{Remark}
\newtheorem{question}[thm]{Question}

\numberwithin{thm}{section}

\newcommand{\adj}{\leftrightarrow}
\newcommand{\adjeq}{\leftrightarroweq}

\DeclareMathOperator{\id}{id}
\def\N{{\mathbb N}}

\DeclareMathOperator{\Fix}{Fix}
\newcommand{\Z}{\mathbb{Z}}
  
\begin{document}
\title{Fixed Point Sets in Digital Topology, 1}
\author{Laurence Boxer
         \thanks{
    Department of Computer and Information Sciences,
    Niagara University,
    Niagara University, NY 14109, USA;
    and Department of Computer Science and Engineering,
    State University of New York at Buffalo
    email: boxer@niagara.edu
    }
\and{P. Christopher Staecker
     \thanks{
     Department of Mathematics, Fairfield 
     University, Fairfield, CT 06823-5195, USA
     email: cstaecker@fairfield.edu}
}
}
\date{}
\maketitle

\begin{abstract}    
    In this paper,
we examine some properties of the
fixed point set of a digitally
continuous function. The digital setting requires new methods that are not analogous to
those of classical topological fixed point theory, and we obtain results that
often differ greatly from standard results in classical topology.

We introduce several measures related to fixed points
for continuous self-maps on
digital images, and study their
properties. Perhaps the most
important of these is
the fixed point spectrum $F(X)$ of a digital image: that is, the set of all numbers that
can appear as the number of fixed points for some continuous self-map. 
We give a complete computation of $F(C_n)$ where $C_n$ is the digital cycle of $n$ points. 
For other digital images, we show that, if $X$ has at least 4 points, then $F(X)$ always contains the numbers 0, 1, 2, 3, and the cardinality of $X$. 
We give several examples, including $C_n$, in which $F(X)$ does not equal $\{0,1,\dots,\#X\}$. 

We examine how fixed point sets are
affected by rigidity, retraction, deformation
retraction, and the formation of
wedges and Cartesian products.
We also study how fixed point sets in digital images can be arranged; e.g.,
in some cases the fixed point set is always connected.
\end{abstract}

\section{Introduction}
Digital images are often used as
mathematical models of real-world objects.
A digital model of the notion of a continuous function,
borrowed from the study of topology, is
often useful for the study of  digital images.
However, a digital image is typically
a finite, discrete point set. Thus,
it is often necessary to study digital
images using methods not directly
derived from topology.
In this paper, we
introduce several
such methods to study
properties of the
fixed point set of a
continuous self-map
on a digital image.

\section{Preliminaries}
Let $\N$ denote the set of natural numbers; 
and $\Z$, the set of integers. $\#X$ will be
used for the number of elements of a set~$X$.

\subsection{Adjacencies}
A digital image is a pair $(X,\kappa)$ where
$X \subset \Z^n$ for some $n$ and $\kappa$ is
an adjacency on $X$. Thus, $(X,\kappa)$ is a graph
for which $X$ is the vertex set and $\kappa$ 
determines the edge set. Usually, $X$ is finite,
although there are papers that consider infinite $X$. Usually, adjacency reflects some type of
``closeness" in $\Z^n$ of the adjacent points.
When these ``usual" conditions are satisfied, one
may consider the digital image as a model of a
black-and-white ``real world" digital image in which
the black points (foreground) are represented by 
the members of $X$ and the white points 
(background) by members of $\Z^n \setminus X$.

We write $x \adj_{\kappa} y$, or $x \adj y$ when
$\kappa$ is understood or when it is unnecessary to
mention $\kappa$, to indicate that $x$ and
$y$ are $\kappa$-adjacent. Notations 
$x \adjeq_{\kappa} y$, or $x \adjeq y$ when
$\kappa$ is understood, indicate that $x$ and
$y$ are $\kappa$-adjacent or are equal.

The most commonly used adjacencies are the
$c_u$ adjacencies, defined as follows.
Let $X \subset \Z^n$ and let $u \in \Z$,
$1 \le u \le n$. Then for points
\[x=(x_1, \ldots, x_n) \neq (y_1,\ldots,y_n)=y\]
we have $x \adj_{c_u} y$ if and only if
\begin{itemize}
    \item for at most $u$ indices $i$ we have
          $|x_i - y_i| = 1$, and
    \item for all indices $j$, $|x_j - y_j| \neq 1$
          implies $x_j=y_j$.
\end{itemize}

The $c_u$-adjacencies are often denoted by the
number of adjacent points a point can have in the
adjacency. E.g.,
\begin{itemize}
\item in $\Z$, $c_1$-adjacency is 2-adjacency;
\item in $\Z^2$, $c_1$-adjacency is 4-adjacency and
      $c_2$-adjacency is 8-adjacency;
\item in $\Z^3$, $c_1$-adjacency is 8-adjacency,
      $c_2$-adjacency is 18-adjacency, and 
      $c_3$-adjacency is 26-adjacency.
\end{itemize}

We will often discuss the \emph{digital $n$-cycle}, the $n$-point image $C_n =\{x_0,\dots,x_{n-1}\}$ in which each $x_i$ is adjacent only to $x_{i+1}$ and $x_{i-1}$, and subscripts are always read modulo $n$.

The literature also contains several adjacencies
to exploit properties of Cartesian products
of digital images. These include the following.

\begin{definition}
\cite{Berge}
Let $(X,\kappa)$ and $(Y, \lambda)$ be digital
images. The {\em normal product adjacency} or
{\em strong adjacency} on $X \times Y$, denoted
$NP(\kappa, \lambda)$, is defined as follows.
Given $x_0, x_1 \in X$, $y_0, y_1 \in Y$ such that
\[p_0=(x_0,y_0) \neq (x_1,y_1)=p_1, \]
we have $p_0 \adj_{NP(\kappa,\lambda)} p_1$ if
and only if one of the following is valid:
\begin{itemize}
    \item $x_0 \adj_{\kappa} x_1$ and $y_0=y_1$, or
    \item $x_0 = x_1$ and $y_0 \adj_{\lambda} y_1$,
          or
    \item $x_0 \adj_{\kappa} x_1$ and
          $y_0 \adj_{\lambda} y_1$.
\end{itemize}
\end{definition}

\begin{thm}
{\rm \cite{BK12}}
\label{NPm+n}
Let $X \subset \Z^m$, $Y \subset \Z^n$. Then
\[(X \times Y, NP(c_m,c_n)) = (X \times Y, c_{m+n}),
\]
i.e., the $c_{m+n}$-adjacency on 
$X \times Y \subset \Z^{m+n}$ coincides with the
normal product adjacency based on $c_m$ and $c_n$.
\end{thm}

Building on the normal product adjacency, we have
the following.

\begin{definition}
{\rm \cite{BxNormal}}
Given $u, v \in \N$, $1 \le u \le v$, and digital
images $(X_i, \kappa_i)$, $1 \le i \le v$, let
$X = \Pi_{i=1}^v X_i$. The adjacency
$NP_u(\kappa_1, \ldots, \kappa_v)$ for $X$ is
defined as follows. Given $x_i, x_i' \in X_i$, let
\[p=(x_1, \ldots, x_v) \neq (x_1', \ldots, x_v')=q. \]
Then $p \adj_{NP_u(\kappa_1, \ldots, \kappa_v)} q$
if for at least 1 and at most $u$ indices $i$ we have
$x_i \adj_{\kappa_i} x_i'$ and for all other
indices $j$ we have $x_j = x_j'$.
\end{definition}

Notice $NP(\kappa, \lambda)= NP_2(\kappa, \lambda)$
\cite{BxNormal}.

\subsection{Digitally continuous functions}
We denote by $\id$ or $\id_X$ the
identity map $\id(x)=x$ for all $x \in X$.

\begin{definition}
{\rm \cite{Rosenfeld, Bx99}}
Let $(X,\kappa)$ and $(Y,\lambda)$ be digital
images. A function $f: X \to Y$ is 
{\em $(\kappa,\lambda)$-continuous}, or
{\em digitally continuous} when $\kappa$ and
$\lambda$ are understood, if for every
$\kappa$-connected subset $X'$ of $X$,
$f(X')$ is a $\lambda$-connected subset of $Y$.
If $(X,\kappa)=(Y,\lambda)$, we say a function
is {\em $\kappa$-continuous} to abbreviate
``$(\kappa,\kappa)$-continuous."
\end{definition}

\begin{thm}
{\rm \cite{Bx99}}
A function $f: X \to Y$ between digital images
$(X,\kappa)$ and $(Y,\lambda)$ is
$(\kappa,\lambda)$-continuous if and only if for
every $x,y \in X$, if $x \adj_{\kappa} y$ then
$f(x) \adjeq_{\lambda} f(y)$.
\end{thm}

\begin{thm}
\label{composition}
{\rm \cite{Bx99}}
Let $f: (X, \kappa) \to (Y, \lambda)$ and
$g: (Y, \lambda) \to (Z, \mu)$ be continuous 
functions between digital images. Then
$g \circ f: (X, \kappa) \to (Z, \mu)$ is continuous.
\end{thm}

A {\em path} is a continuous
function $r: [0,m]_{\Z} \to X$. 

We use the following notation. For a
digital image $(X,\kappa)$,
\[ C(X,\kappa) = \{f: X \to X \, | \,
   f \mbox{ is continuous}\}.
\]

\begin{definition}
{\rm (\cite{Bx99}; see also \cite{Khalimsky})}
\label{htpy-2nd-def}
Let $X$ and $Y$ be digital images.
Let $f,g: X \rightarrow Y$ be $(\kappa,\kappa')$-continuous functions.
Suppose there is a positive integer $m$ and a function
$h: X \times [0,m]_{\Z} \rightarrow Y$
such that

\begin{itemize}
\item for all $x \in X$, $h(x,0) = f(x)$ and $h(x,m) = g(x)$;
\item for all $x \in X$, the induced function
      $h_x: [0,m]_{\Z} \rightarrow Y$ defined by
          \[ h_x(t) ~=~ h(x,t) \mbox{ for all } t \in [0,m]_{\Z} \]
          is $(c_1,\kappa')-$continuous. That is, $h_x(t)$ is a path in $Y$.
\item for all $t \in [0,m]_{\Z}$, the induced function
         $h_t: X \rightarrow Y$ defined by
          \[ h_t(x) ~=~ h(x,t) \mbox{ for all } x \in  X \]
          is $(\kappa,\kappa')-$continuous.
\end{itemize}
Then $h$ is a {\em digital $(\kappa,\kappa')-$homotopy between} $f$ and
$g$, and $f$ and $g$ are {\em digitally $(\kappa,\kappa')-$homotopic in} $Y$, denoted
$f \simeq_{\kappa,\kappa'} g$ or $f \simeq g$ when
$\kappa$ and $\kappa'$ are understood.
If $(X,\kappa)=(Y,\kappa')$, we say $f$ and $g$ are
{\em $\kappa$-homotopic} to abbreviate
``$(\kappa,\kappa)$-homotopic" and write
$f \simeq_{\kappa} g$ to abbreviate
``$f \simeq_{\kappa,\kappa} g$".


If there is a $\kappa$-homotopy between
$\id_X$ and a constant map, we say $X$ is {\em $\kappa$-contractible}, or just {\em contractible}
when $\kappa$ is understood.
\end{definition}

\begin{definition}
Let $A \subseteq X$. A $\kappa$-continuous
function $r: X \to A$ is a {\em retraction}, and
{\em $A$ is a retract of $X$}, if $r(a)=a$ for
all $a \in A$. If such a map $r$ satisfies
$i \circ r \simeq_{\kappa} \id_X$ where $i: A \to X$ is the
inclusion map, then $r$ is a
{\em $\kappa$-deformation retraction} and
$A$ is a {\em $\kappa$-deformation retract} of $X$.
\end{definition}

A topological space $X$ has the {\em fixed point property (FPP)} if every continuous $f: X \to X$ has
a fixed point. A similar 
definition has appeared in
digital topology: a digital image $(X,\kappa)$ has 
the {\em fixed point property (FPP)} if every 
$\kappa$-continuous $f: X \to X$ has
a fixed point. However, this
property turns out to be trivial, in the sense of
the following.

\begin{thm}
{\rm \cite{BEKLL}}
\label{BEKLLFPP}
A digital image $(X,\kappa)$ has the FPP if and
only if $\#X=1$.
\end{thm}

The proof~\cite{BEKLL} of
Theorem~\ref{BEKLLFPP} was
due to the establishment
of the following.

\begin{lem}
\label{howBEKLLFPP}
Let $(X,\kappa)$ be a
digital image, where
$\#X > 1$. Let
$x_0,x_1 \in X$ be such that
$x_0 \adj_{\kappa} x_1$. Then the
function $f: X \to X$
given by $f(x_0)=x_1$ and
$f(x)=x_0$ for $x \neq x_0$
is $\kappa$-continuous and
has 0 fixed points.
\end{lem}

A function $f: (X,\kappa) \to (Y,\lambda)$ is
an {\em isomorphism} (called a {\em homeomorphism}
in~\cite{Bx94}) if $f$ is a continuous bijection
such that $f^{-1}$ is continuous.

\section{Rigidity}
We will say a function $f:X\to Y$ is \emph{rigid} when no continuous map is homotopic to $f$ except $f$ itself. This generalizes a definition in \cite{hmps}. When the 
identity map $\id:X\to X$ is rigid, we say $X$ is rigid. 

Many digital images are rigid, though it can be difficult to show directly that a given example is rigid. A computer search described in \cite{stae15} has shown that no rigid images in $\Z^2$ with 4-adjacency exist having fewer than 13 points, and no rigid images in $\Z^2$ with 8-adjacency exist having fewer than 10 points. We will demonstrate some methods for showing that a given image is rigid. 
For example, the digital image in Figure \ref{rigidwedge} is rigid, 
as shown below in Example \ref{wedgeOfCycles}. 

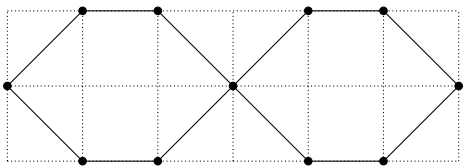
\begin{figure}
\[
\begin{tikzpicture}
\draw[densely dotted] (-3,-1) grid (3,1);
\foreach \x in {-2,-1,2,1} {
 \node at (\x,-1) [vertex] {};
 \node at (\x,1) [vertex] {};
 }
\foreach \x in {-3,0,3}
 \node at (\x,0) [vertex] {};
\draw (-3,0)--(-2,1)--(-1,1)--(0,0)--(1,1)--(2,1)--(3,0);
\draw (-3,0)--(-2,-1)--(-1,-1)--(0,0)--(1,-1)--(2,-1)--(3,0);
\end{tikzpicture}
\]
\caption{A rigid image in $\Z^2$ with 8-adjacency\label{rigidwedge}}
\end{figure}

An immediate consequence of
the definition of rigidity is the following.

\begin{prop}
Let $(X,\kappa)$ be a rigid digital image
such that $\#X > 1$. Then $X$ is not
$\kappa$-contractible.
\end{prop}

Rigidity of functions is preserved when composing with an isomorphism, as the following theorems demonstrate:

\begin{thm}
\label{isoOfRigidIsRigid}
Let $f:X\to Y$ be rigid and $g:Y \to Z$ be an isomorphism. Then $g \circ f:X \to Z$ is rigid.
\end{thm}

\begin{proof}
Suppose otherwise. Then there is a homotopy
$h: X \times [0,m]_{\Z} \to Z$ from
$g \circ f$ to a map $G: X \to Z$ such that
$g \circ f \neq G$. Then by
Theorem~\ref{composition},
$g^{-1} \circ h: X \times [0,m]_{\Z} \to X$ is
a homotopy from $f$ to $g^{-1} \circ G$, and
since $g^{-1}$ is one-to-one,
$f \neq g^{-1} \circ G$. This
contradiction of the assumption that $f$ is rigid
completes the proof.
\end{proof}

\begin{thm}
\label{RigidOfIsoIsRigid}
Let $f:X\to Y$ be rigid. Let $g: W \to X$ be an
isomorphism. Then $f \circ g$ is rigid.
\end{thm}

\begin{proof}
Suppose otherwise. Then there is a homotopy
$h: W \times [0,m]_{\Z} \to Y$ from $f \circ g$
to some $G: W \to Y$ such that $G \neq f \circ g$.
Thus, for some $w \in W$, $G(w) \neq f \circ g(w)$.
Now consider the function
$h': X \times [0,m]_{\Z} \to Y$ defined by
$h'(x,t) = h(g^{-1}(x),t)$. By
Theorem~\ref{composition}, $h'$ is a homotopy
from $f  \circ g \circ g^{-1} = f$ to
$G \circ g^{-1}$. Since
\[ f  \circ g (w) \neq G(w) = (G \circ g^{-1})(g(w)),
\]
the homotopic functions $f$ and $G \circ g^{-1}$
differ at $g(w)$, contrary to the assumption
that $f$ is rigid. The assertion follows.
\end{proof}

As an immediate corollary, we obtain:
\begin{cor}\label{rigidiso}
If $f:X\to Y$ is an isomorphism and one of $X$ and $Y$ are rigid, then $f$ is rigid.
\end{cor}
\begin{proof}
In the case where $X$ is rigid, the identity map $\id_X$ is rigid. Then by Theorem \ref{RigidOfIsoIsRigid} we have $f \circ \id_X= f$ is rigid. In the case where $Y$ is rigid, similarly by Theorem \ref{isoOfRigidIsRigid} we have $\id_Y \circ f=f$ is rigid.
\end{proof}

The  corollary above can be stated equivalently as follows:
\begin{cor}
A digital image $X$ is rigid if and only if every
digital image $Y$ that is isomorphic to $X$ is
rigid.
\end{cor}

It is easy to see that no digital image in $\Z$ is rigid:
\begin{prop}
If $X \subset \Z$ is a connected digital image with $c_1$ adjacency and $\#X>1$, then $X$ is not rigid.
\end{prop}

\begin{proof}
A connected subset of $\Z$ having more than
one point takes one of the forms
\[ [a,b]_{\Z},~\{z \in \Z \, | \, z \ge a\},~\{z \in \Z \, | \, z \le b\},~\Z.
\]
In all of these cases, it is easily seen
that there is a deformation retraction
of $X$ to a proper subset of $X$. 
Therefore, $X$ is not rigid.
\end{proof}

%

We also show that a normal product of images is rigid if and only if all of its factors are rigid.
\begin{thm}
\label{rigidProd}
Let $(X_i,\kappa_i)$ be digital images for each $1 \le i \le v$, and $(\prod_{i=1}^v X_i, \kappa)$ be the product image, where $\kappa = NP_u(\kappa_1,\ldots, \kappa_v)$ for some $1 \le u \le v$. Then $X$ is rigid if and only if $X_i$ is rigid for each $i$.
\end{thm}

\begin{proof}
First we assume $X$ is rigid, and we will show that $X_i$ is rigid for each $i$. For some $i$, let
$h_i: X_i \times [0,m]_{\Z} \to X_i$ be a 
$\kappa_i$-homotopy
from $\id_{X_i}$ to $f_i: X_i \to X_i$. Without loss
of generality we may assume $m = 1$, and we will show that $h_i(x_i,1)=x_i$, and thus $f_i = \id_{X_i}$. 
The function
$h: X \times [0,1]_{\Z} \to X$ defined by
\[ h(x_1, \ldots, x_v, t) = (x_1, \ldots, x_{i-1},h_i(x_i,t), x_{i+1}, \dots, x_v),
\]
is a homotopy. Since $X$ is rigid we must have $h(x_1, \ldots, x_v, 1) = \id_X$, and this means $h_i(x_i,1)=x_i$ as desired.

Now we prove the converse: assume $X_i$ is rigid, and we will show $X$ is rigid.
Let $J_i: X_i \to X$ be the function
\[ J_i(x) = (x_1, \ldots, x_{i-1}, x, x_{i+1}, \ldots, x_n).
\]
Let $p_i: X \to X_i$ be the projection function,
\[p_i(y_1,\ldots, y_n)=y_i.
\]
Then $J_i$ is $(\kappa_i, \kappa)$-continuous, and $p_i$ is $(\kappa, \kappa_i)$-continuous~\cite{BxNormal,BxAlt}.

For the sake of a contradiction, suppose $X$ is not rigid. Then there is a homotopy
$h: (X,\kappa) \times [0,m]_{\Z} \to X$ between $\id_X$ and a function
$g$ such that for some $y = (y_1,\ldots,y_v) \in X$, $g(y) \neq y$.
Then for some index $j$, $p_j(y) \neq p_j(g(y))$. Then the function
$h': X_j \times [0,m]_{\Z} \to X_j$ defined
by
$h'(x,t)= p_j(h(J_j(x), t))$,
is a homotopy from $\id_{X_j}$ to a function
$g_j$, with 
\[ g_j(y_j) = p_j(h(J_j(y_j), m)) = p_j(h(y,m))
= p_j(g(y)) \neq p_j(y)= y_j,
\]
contrary to the assumption that $X_i$ is rigid.
We conclude that $X$ is rigid.
\end{proof}

We have a similar result when $X$ is a disjoint union of digital images. Let $X$ be a digital image of the form $X=A\cup B$ where $A$ and $B$ are disjoint and no point of $A$ is adjacent to any point of $B$. We say $X$ is the disjoint union of $A$ and $B$, and we write $X=A\sqcup B$. 
\begin{thm}
Let $X = A\sqcup B$. Then $X$ is rigid if and only if $A$ and $B$ are rigid.
\end{thm}
\begin{proof}
First we assume that $X$ is rigid, and we will show that $A$ is rigid. (It will follow from a similar argument that $B$ is rigid.) Let $f:A\to A$ be any self-map homotopic to $\id_A$, and we will show that $f=\id_A$. Define $g:X\to X$ by
\[ g(x) = \begin{cases}
f(x) \quad &\text{ if $x\in A$;} \\
x \quad &\text{ if $x \in B$.} \end{cases} \]
Then $g$ is continuous and homotopic to $\id_X$, and since $X$ is rigid we must have $g=\id_X$, which means that $f=\id_A$.

Now for the converse, assume that $A$ and $B$ are both rigid. Take some self-map $f:X\to X$ homotopic to $\id_X$, and we will show that $f=\id_X$. Since $f$ is homotopic to the identity, we must have $f(A)\subseteq A$ and $f(B)\subseteq B$. This is because there will always be a path from any point $x$ to $f(x)$ given by the homotopy from $\id_X$ to $f(x)$. 
Thus if $x\in A$ we must also have $f(x)\in A$ since there are no paths from points of $A$ to points of $B$.

Since $f(A)\subseteq A$ and $f(B)\subseteq B$, there are well-defined restrictions 
$f_A:A\to A$ and $f_B:B\to B$, and the homotopy from $\id_X$ to $f$ induces homotopies from $\id_A$ to $f_A$ and $\id_B$ to $f_B$. Since $A$ and $B$ are rigid we must have $f_A=\id_A$ and $f_B=\id_B$, and thus $f=\id_X$ as desired.
\end{proof}

Since every digital image is a disjoint union of its connected components, we have:
\begin{cor}\label{componentsrigid}
A digital image $X$ is rigid if and only if every connected component of $X$ is rigid.
\end{cor}

Let $X$ be some digital image of the form $X=A \cup B$, where $A \cap B$ is a single point $x_0$, and no point of $A$ is adjacent to any point of $B$ except $x_0$. We say $X=A \cup B$ is the {\em wedge
of $A$ and $B$}, denoted $X =A\wedge B$, and $x_0$ is called the \emph{wedge point} of $A\wedge B$. We have the following.

\begin{thm}
\label{wedgeThm}
If $X=A\wedge B$ and $A$ and $B$ are rigid, then $X$ is rigid.
\end{thm}

\begin{proof}
Let $x_0$ be the wedge point of $A\wedge B$, and let $A_0$ and $B_0$ be the components of $A$ and $B$ which include $x_0$. If $\#A_0=1$ or $\#B_0=1$, then the components of $A\wedge B$ are in direct correspondence to the components of $A$ and $B$ and the result follows by Corollary \ref{componentsrigid}. Thus we assume $\#A_0>1$ and $\#B_0>1$.

Let $h: A \wedge B \times [0,m]_{\Z} \to A \wedge B$
be a homotopy such that $h(x,0)=x$ for all
$x \in A \wedge B$. Without loss of generality,
$m=1$. If the induced map $h_1$ is not
$\id_{X}$ then there is a point 
$x' \in X$ such that 
$h_1(x')=h(x',1) \neq x'$. Without loss of
generality, $x' \in A$. Let
$p_A: X \to A$ be the projection
\[ p_A(x) = \left \{ \begin{array}{ll}
          x & \mbox{ for } x \in A; \\
          x_0 & \mbox{ for } x \in B.
   \end{array} \right .
\]

Since $p_A \circ h$ is a homotopy from $\id_A$ to 
$p_A \circ h_1$, and $A$ is rigid, we have
\begin{equation}
\label{idA}
p_A \circ h_1 = \id_A.
\end{equation}
Were 
$h_1(x') \in A$ then it would follow that
\[ h_1(x') = p_A \circ h_1(x') = x', 
          \]
contrary to our choice of $x'$. Therefore
we have $h_1(x') \in B \setminus \{x_0\}$.
But $x' \adj h_1(x')$, so $x'=x_0$.

Since $A_0$ is connected and has more than 1 point,
there exists $x_1 \in A$ such that
$x_1 \adj_{\kappa} x_0$. By the continuity of $h_1$
and choice of $x_0$, we
must therefore have $h_1(x_1)=x_0$, and therefore
$p_A \circ h_1(x_1)=p_A(x_0)=x_0$. 
This contradicts statement~(\ref{idA}), so the
assumption that $h_1$ is not $\id_{X}$ 
is incorrect, and the assertion follows.
\end{proof}

A {\em loop} is a continuous
function $p: C_m \to X$.

The converse of Theorem~\ref{wedgeThm} is not
generally true. In \cite{hmps} it was mentioned (without proof) that a wedge of two long cycles is in general rigid. We give a specific example:

\begin{exl}
\label{wedgeOfCycles}
Let $A$ and $B$ be non-contractible simple closed 
curves. Then $A$ and $B$ are non-rigid~\cite{hmps}.
However, $X=A \wedge B$ is rigid. E.g., using $c_2=8$-adjacency in $\Z^2$, let $A = $
\[ \{a_0=(0,0), a_1=(1,-1), a_2=(2,-1),
         a_3=(3,0), a_4=(2,1), a_5=(1,1)\}
\]
and let $B = $
\[ \{b_0=a_0, b_1=(-1,-1), b_2=(-2,-1),
         b_3=(-3,0), b_4=(-2,1), b_5=(-1,1)\}.
\]
By continuity, if there is a homotopy 
$h: X \times [0,m]_{\Z} \to
X$ - without loss of generality, $m=1$ -
such that $h_0 = \id_{X}$ and 
$h(x,1) \neq x$, then $h$
``pulls"~\cite{hmps} every point of $A$
or of $B$ and therefore
``breaks" one of the loops of $X$, a contradiction since
``breaking" a loop is a discontinuity. 
Thus no such homotopy exists. $\Box$
\end{exl}



\section{Homotopy fixed point spectrum}
The paper \cite{bs19} gave a
brief treatment of homotopy-invariant fixed point theory, defining two quantities $M(f)$ and $X(f)$, respectively 
the minimum and maximum possible number of fixed points among all maps homotopic to $f$. 
When $f:X\to X$, clearly we will have:
\[ 0 \le M(f) \le X(f) \le \#X. \]
We will see in the examples below that any one of these inequalities can be strict in some cases, or equality in some cases.

More generally, for some map $f:X\to X$, we may consider the following set $S(f)$, which we call the \emph{homotopy fixed point spectrum} of $f$:
\[ S(f) = \{ \# \Fix(g) \mid g \simeq f \} \subseteq \{0,\dots,\#X\}. \]

An immediate consequence
of Lemma~\ref{howBEKLLFPP}:

\begin{cor}
\label{0inSpectrum}
Let $(X,\kappa)$ be a
connected digital image, 
where $\#X > 1$. Then
$0 \in S(c)$, where
$c \in C(X,\kappa)$ is a 
constant map.
\end{cor}

We can also consider the \emph{fixed point spectrum} of $X$, defined as:
\[ F(X) = \{ \#\Fix(f) \mid f:X\to X \text{ is continuous} \}  \]

\begin{remark}
{\rm The following assertions are immediate
consequences of the relevant definitions.}
\begin{itemize}
\item If $X$ is a digital image of only one point, then $F(X) = \{1\}$.
\item If $f: X \to X$ is rigid, then
      $S(f) = \{\#\Fix(f)\}$. If $X$ is rigid, then $S(\id) = \{\#X\}$.
\end{itemize}
Since every image $X$ has a constant map and an identity map, we always have:
\[ \{1,\#X\} \subseteq F(X). \]
\end{remark}


The number of fixed points is always preserved by isomorphism:

\begin{lem}
\label{isoMatches}
Let $X$ and $Y$ be isomorphic digital images.
Let $f: X \to X$ be continuous. Then there is
a continuous $g: Y \to Y$ such that
$\#\Fix(f) = \#\Fix(g)$.
\end{lem}

\begin{proof}
Let $G: X \to Y$ be an isomorphism.
Let $A = Fix(f)$.
Since $G$ is one-to-one, $\#G(A) = \#A$. Let
$g: Y \to Y$ be defined by 
$g=G \circ f \circ G^{-1}$. For $y_0 \in G(A)$,
let $x_0=G^{-1}(y_0)$. Then
\[ g(y_0) = G \circ f \circ G^{-1}(y_0) =
   G \circ f(x_0) = G(x_0) = y_0.
\]
Let $B = \Fix(g)$
It follows that $G(A) \subseteq B$, so
$\#A \le \#B$.

Similarly, let $y' \in B$ and let
$x' = G^{-1}(y')$. Then
\[ f(x') = G^{-1} \circ g \circ G(x') =
   G^{-1} \circ g(y') = G^{-1}(y') = x'.
\]
It follows that $G^{-1}(B) \subseteq A$, so
$\#B \le \#A$.

Thus, $\#\Fix(f) = \#\Fix(g)$.
\end{proof}

As an immediate consequence, we have the following.

\begin{cor}
Let $X$ and $Y$ be isomorphic digital images.
Then $F(X)=F(Y)$.
\end{cor}

There is a certain regularity to the fixed point spectrum for connected digital images. When $X$ has only a single point, we have already remarked that $F(X)=\{1\}$. For images of more than 1 point, we will show that $F(X)$ always includes 0, 1, and $\#X$, and, provided the image is large enough, the set $F(X)$ also includes 2 and 3. 

The following statements hold for connected images. We discuss the fixed point spectrum of disconnected images in terms of their connected components in Theorem \ref{duF} and its corollary.
We begin with a simple lemma:

\begin{lem}\label{nbhdF(X)}
Let $X$ be any connected digital image with $\#X > 1$. Let $x_0 \in X$, and let $0\le k \le \#N^*(x_0)$. Then $k \in S(c) \subseteq F(X)$, where $c$ is a constant map.
\end{lem}
\begin{proof}
By Corollary~\ref{0inSpectrum},
a constant map is homotopic to a map with no fixed points, so $0\in S(c)$ as desired.

For $k>0$, let $n=\#N^*(x_0)$ and write 
\[ N^*(x_0) = \{x_0,x_1,\dots,x_{n-1}\}. \]
Then define $f:X\to X$ by:
\[ f(x) = \begin{cases} x & \text{ if $x=x_i$ for some $i<k$,} \\
x_0 &\text{ otherwise.}
\end{cases} \]
Then $f$ is continuous with $\Fix(f) = \{x_0,\dots,x_{k-1}\}$ and thus $k\in F(X)$. Furthermore, $f$ is homotopic to the constant map at $x_0$, and so in fact $k\in S(c)$. 
\end{proof}



\begin{thm}\label{0123}
Let $X$ be a connected digital image, and let $c:X\to X$ be any constant map. If $\#X 
\ge 2$ then
\[ \{0,1,2\}\subseteq S(c). \]
If $\#X \ge 3$, then
\[ \{0,1,2,3\} \subseteq S(c). \]
\end{thm}
\begin{proof}
If $\#X = 2$, then $X$ consists simply of two adjacent points. Thus $\#N^*(x)=2$ for each $x\in X$, and so Lemma \ref{nbhdF(X)} implies that $\{0,1,2\}\subseteq S(c)$.

When $\#X\ge 3$, there must be some $x\in X$ with $\#N^*(x) \ge 3$. (Otherwise the image would consist only of disjoint pairs of adjacent points, which would not be connected.) Thus by Lemma \ref{nbhdF(X)} we have $\{0,1,2,3\}\subseteq S(c)$. 
\end{proof}

Since we always have $\#X \in S(\id)$ and 
\[S(c)\cup S(\id) \subseteq F(X) \subseteq \{0,1,\dots,\#X\},\]
the theorem above directly gives:

\begin{cor}\label{0123X}
Let $X$ be a connected digital image. If $\#X = 2$ then
\[ F(X) = \{0,1,2\}. \]
If $\#X > 2$, then
\[ \{0,1,2,3,\#X\} \subseteq F(X). \]
\end{cor}

We have already seen that $\#X \in F(X)$ in all cases. There is an easy condition that determines whether or not $\#X-1\in F(X)$. 

\begin{lem}
\label{1off}
Let $X$ be connected with $n=\#X >1$. Then $n-1 \in F(X)$ if and only if there are distinct points $x_1,x_2 \in X$ with 
$N(x_1) \subseteq N^*(x_2)$.
\end{lem}
\begin{proof}
Suppose there are points
$x_1,x_2 \in X$, $x_1 \neq x_2$, such that
$N(x_1)  \subseteq N^*(x_2)$. 
Then the map 
\[  f(x_1)=x_2,\quad f(x)=x \text{ for all } x\neq x_1, \]
is a self-map on $X$ with exactly $n-1$ fixed points. That
$f$ is continuous is seen as follows. Suppose
$x,x' \in X$ with $x \adj x'$.
\begin{itemize}
    \item If $x_1 \not \in \{x,x'\}$, then
          \[ f(x)=x\adj x' = f(x'). \]
    \item If, say, $x = x_1$, then 
          $x' \in N(x_1) \subseteq N^*(x_2)$, so
          \[ f(x')=x' \adjeq x_2 = f(x_1). \]
\end{itemize}
Thus $f$ is continuous, and we conclude 
$n-1 \in F(X)$.

Now assume that $n-1 \in F(X)$. Thus there is some continuous self-map $f$ with exactly $n-1$ fixed points. Let $x_1$ be the single point not fixed by $f$, and let $x_2 = f(x_1)$. Then let $x \in X$
with $x \adj x_1$. Then
\[ x=f(x) \adjeq f(x_1) = x_2, \]
so $N(x_1) \subseteq N^*(x_2)$.
\end{proof}

Lemma~\ref{1off} can be used to show that a large class of digital images will satisfy $n-1 \not \in F(X)$. For example when $X=C_n$ for $n>4$, no 
$N(x_i)$ is contained in
$N^*(x_j)$ for $j \neq i$. Thus we have:
\begin{cor}
Let $n>4$. Then $n-1 \not \in F(C_n)$.
\end{cor}
In particular this means that 
$4\not \in F(C_5)$, so the result of 
Theorem~\ref{0123X} cannot in general be 
improved to state that $4\in F(X)$ for 
all images of more than 4 points.

\section{Pull indices}
Let $\overline \Fix(f)$ be the complement of the fixed point set, that is,
\[ \overline \Fix(f) = \{ x \in X \mid f(x)\neq x\}. \]
When $f(x)\neq x$, we say \emph{$f$ moves $x$}.
\begin{definition}
Let $(X,\kappa)$ be a digital image with $\#X>1$ and 
let $x \in X$. 
The {\em pull index of $x$},
$P(x)$ or $P(x,X)$ or $P(x,X,\kappa)$, is
\[ P(x) = \min \{\#\overline\Fix(f) \mid f:X\to X \text{ is continuous and } f(x)\neq x\}. \]
\end{definition}

When $f(x)\neq x$, the set $\overline\Fix(f)$ always contains at least the point $x$, and so $P(x)\ge 1$ for any $x$ that is moved by
some~$f$.

\begin{exl}
Let $X=[1,3]_{\Z}$ with $c_1$-adjacency.

To compute $P(3)$, consider the function $f(x)=\min\{x,2\}$. This is continuous, not the identity, and $\Fix(f) = \{1,2\}$, and thus $P(3) = 1$. Similarly we can show that $P(1)=1$.

But we have $P(2)=2$, since any continuous self-map~$f$ 
on~$X$ that moves 2 must also move at least
one other point: if $f(2)=1$ we must have $f(3) \in \{1,2\}$, and if $f(2)=3$ we must have $f(1)\in\{2,3\}$.
\end{exl}

\begin{prop}
\label{pullThm}
Let $(X,\kappa)$ be a connected digital image with
$n=\#X >1$. Let $m \in \N$, $1 \le m \le n$. 
Suppose, for all $x \in X$, we have
$P(x) \ge m$ . Then
\[ F(X) \cap \{i\}_{i=n-m+1}^{n-1} = \emptyset.
\]
\end{prop}

\begin{proof}
By hypothesis,
$f \in C(X,\kappa) \setminus \{\id_X\}$ implies $f$ moves
at least $m$ points, hence 
$\#\Fix(f) \le n-m$. 
The assertion follows.
\end{proof}

\begin{thm}
\label{F(X)andP(x)}
Let $(X,\kappa)$ be a connected 
digital image with $n=\#X > 1$.
The following are equivalent.

{\rm 1)} $n-1 \in F(X)$.

{\rm 2)} There are distinct
   $x_1,x_2 \in X$ such that
   $N(x_1) \subseteq N^*(x_2)$.
   
{\rm 3)} There exists $x \in X$
          such that $P(x)=1$.
\end{thm}

\begin{proof}
$1) \Leftrightarrow 2)$ is shown in Lemma~\ref{1off}.

$1) \Leftrightarrow 3)$: We have
$n-1 \in F(X)$ $\Leftrightarrow$
there exists $f \in C(X)$ with
exactly~$n-1$ fixed points, i.e.,
the only $x \in X$ not fixed 
by $f$ has $P(x)=1$.
\end{proof}

The following generalizes
$1) \Rightarrow 3)$ of
Theorem~\ref{F(X)andP(x)}.

\begin{prop}
Let $(X,\kappa)$ be a connected 
digital image with $n=\#X > 1$.
Let $k \in [1,n-1]_{\Z}$. Then
$k \in F(X)$ implies there
exist distinct $x_1, \ldots, x_{n-k} \in X$ such that
$P(x_i) \le n-k$.
\end{prop}

\begin{proof}
$k \in F(X)$ implies there
exists $f \in C(X)$ with exactly
$k$ fixed points, hence distinct $x_1, \ldots, x_{n-k} \in X$
such that $x_i \not \in \Fix(f)$.
Thus for each~$i$, the members
of $\Fix(f)$ are not pulled by~$f$
and~$x_i$. Thus $P(x_i) \le n-k$.
\end{proof}

\section{Retracts}
In this section, we study
how retractions interact
with fixed point spectra.

\begin{thm}
{\rm \cite{Bx94}}
\label{extension}
Let $(X,\kappa)$ be a digital image and 
let $A \subseteq X$. Then $A$ is a retract 
of $X$ if and only if for every continuous
$f: (A,\kappa) \to (Y,\lambda)$ there 
is an extension of $f$ to a continuous
$g: (X,\kappa) \to (Y,\lambda)$.
\end{thm}

In the proof of Theorem~\ref{extension}, 
an extension of $f$ is obtained by using 
$g=r \circ f$, where $r: X \to A$ is 
a retraction. We use this in the proof of 
the next assertion.

\begin{thm}
\label{retractF}
Let $A$ be a retract of $(X,\kappa)$. Then
$F(A) \subseteq F(X)$.
\end{thm}

\begin{proof}
Let $f: A \to A$ be $\kappa$-continuous.
Let $r: X \to A$ be a $\kappa$-retraction.
Let $i: A \to X$ be the inclusion function. By
Theorem~\ref{composition}, 
$G=i \circ f \circ r: X \to X$ is continuous.
Further,  $G(x)=f(x)$ if and only if
$x \in A$, so $\Fix(G) = \Fix(f)$. Since $f$ was taken arbitrarily, the 
assertion follows.
\end{proof}

\begin{remark}
{\rm
We do not have an analog to Theorem~\ref{retractF}
by replacing fixed point spectra by spectra of
identity maps. E.g, in Example~\ref{wedgeOfCycles}
we have $\{0,\#A\} \subseteq S(\id_A)$, and
$A$ is a retract of $X$,
but $X$ is rigid, so $S(\id_X) = \{\#X\}$. However,
we have the following
Corollaries~\ref{FforDeformation} and~\ref{FforInterval}.}
\end{remark}

\begin{cor}
\label{FforDeformation}
Let $A$ be a deformation retract of $X$. Then
$S(\id_A) \subseteq S(\id_X) \subseteq F(X)$. In
particular, $\#A \in S(\id_X)$. 
\end{cor}

\begin{cor}
\label{FforInterval}
Let $a,b \in \Z$, $a < b$. Then
\[S(\id_{[a,b]_{\Z}}, c_1) = F([a,b]_{\Z}, c_1) = \{0,1,\dots,b-a+1\}.
\]
\end{cor}

\begin{proof}
Since $a<b$ and $[a,b]_{\Z}$ is $c_1$-contractible,
it follows from Theorem~\ref{BEKLLFPP} that
$0 \in S(\id_{[a,b]_{\Z}}, c_1)$.
Since for each $d \in [a,b]_{\Z}$ 
there is a $c_1$-deformation of
$[a,b]_{\Z}$ to $[a,d]_{\Z}$, it follows from Corollary~\ref{FforDeformation} that
$\#[a,d]_{\Z} \in  S(\id_{[a,b]_{\Z}}, c_1)$. Thus,
\[F([a,b]_{\Z}, c_1) \subseteq \{i\}_{i=0}^{b-a+1} = S(\id_{[a,b]_{\Z}}, c_1) \subseteq F([a,b]_{\Z}, c_1).
\]
The assertion follows.
\end{proof}

We can generalize this result about intervals to a two-dimensional box in $\Z^2$.

\begin{thm}
\label{Sbox}
Let $X = [1,a]_{\Z} \times [1,b]_{\Z}$, with adjacency 
$\kappa \in \{c_1,c_2\}$. Then
\[ S(\id_X) = F(X) = \{0,1,\dots,ab\}
\]
\end{thm}
\begin{proof}
All self-maps on $[1,a]_\Z \times [1,b]_\Z$ are homotopic to the identity, so it suffices only to show that $F(X) = \{0,1,\dots,ab\}$. 
The proof is by induction on $b$. For $b=1$, our image $X$ is isomorphic to the one-dimensional image $([1,a]_\Z, c_1)$. Thus by Theorem \ref{FforInterval} we have
\[ F(X) = \{0,1,\dots,a\} = \{0,1,\dots,ab\} \]
as desired.

For the inductive step, first note that $[1,a]_\Z \times [1,b-1]_\Z$ is a retract of $X$ (using either $\kappa=c_1$ or $c_2$). Thus by induction and Theorem \ref{retractF} we have
\[ \{0,1,\dots,a(b-1) \} \subseteq F(X). \]
It remains only to show that 
\[ \{a(b-1)+1, a(b-1)+2, \dots, ab\} \subseteq F(X). \]
We do this by exhibiting a family of self-maps of $X$ having these numbers of fixed points.

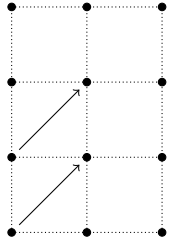
\begin{figure}
\[
\begin{tikzpicture}
\draw[densely dotted] (1,1) grid (3,4);
\foreach \x in {1,...,3} {
 \foreach \y in {1,...,4} {
  \node at (\x,\y) [vertex] {};
 }
}

\newcommand{\pad}{.1}
\foreach \x in {1,2} {
 \draw[->] (1+\pad,\x+\pad) -- (2-\pad,\x+1-\pad);
} 
\end{tikzpicture}
\]
\caption{The map $f_t$ from Theorem \ref{Sbox}, pictured in the case $t=2$. All points are fixed except those with arrows indicating where they map to.\label{f_tfig}}
\end{figure}

Let $t\in \{0,\dots, b-1\}$, and define $f_t:X\to X$ as follows:
\[ f_t(x,y) = \begin{cases} (x,y) \quad &\text{ if $x>1$ or $y>t$}, \\
(x+1,y+1) &\text{ if $x=1$ and $y\le t$} \end{cases} \]
See Figure \ref{f_tfig} for a pictorial depiction of $f_t$. This $f_t$ is well-defined and both $c_1$- and $c_2$-continuous for each $t\in \{0,\dots, b-1\}$ and has $ab-t$ fixed points. Thus we have
\[ \{ab, ab-1, \dots, ab-(b-1) = a(b-1)+1\} \subseteq F(X) \]
as desired.
\end{proof}

\section{Cartesian products and disjoint unions}
In the following, assume $A_i \subset \N$,
$1 \le i \le v$. Define
\[ \bigotimes_{i=1}^v A_i = \left\{\prod_{i=1}^v a_i \mid a_i \in A_i \right\}.
\]
and
\[ \bigoplus_{i=1}^v A_i = \left\{\sum_{i=1}^v a_i \mid a_i \in A_i \right\}.
\]
If $f_i: X_i \to Y_i$, let
$\Pi_{i=1}^v f_i: \Pi_{i=1}^v X_i \to \Pi_{i=1}^v Y_i$
be the product function defined by
\[ \Pi_{i=1}^v f_i(x_1, \ldots x_v) =
   (f_1(x_1), \ldots, f_v(x_v)) \mbox{ for }
   x_i \in X_i.
\]

\begin{thm}
\label{bigOandNP}
Suppose $(X_i,\kappa_i)$ is a digital image, 
$1 \le i \le v$. Let $X = \Pi_{i=1}^v X_i$.
Then
$\bigotimes_{i=1}^v F(X_i,\kappa_i) \subseteq F(X,NP_v(\kappa_1, \ldots, \kappa_v))$.
\end{thm}

\begin{proof}
Let $f_i: X_i \to X_i$ be $\kappa_i$-continuous. Let
$X = \Pi_{i=1}^v X_i$. Then the
product function
\[ f = \Pi_{i=1}^v f_i(x_1, \ldots x_v): X \to X
\]
is $NP_v(\kappa_1, \ldots, \kappa_v)$-continuous~\cite{BxNormal}.
If $A_i=\{y_{i,j}\}_{j=1}^{p_i}$ is the set of 
distinct fixed points of $f_i$, then each point
$(y_{1,j_1}, \ldots, y_{v,j_v})$, for
$1 \le j_i \le p_i$, is a fixed point of $f$.
The assertion follows.
\end{proof}

We note that the conclusion of Theorem \ref{bigOandNP} cannot in general be strengthened to say that $\bigotimes_{i=1}^v F(X_i) = F(X)$. For example,
if $X = [1,3]_\Z \times [1,3]_\Z$, we have 
$F(X) = \{0,1,\dots,9\}$ by Theorem \ref{Sbox}, but 
\[ F([1,3]_\Z) \otimes F([1,3]_\Z) = \{0,1,2,3\} \otimes \{0,1,2,3\} = \{0,1,2,3,4,6,9\}. \]

We do have a similar result, this time with equality, for a disjoint union of digital images.
\begin{thm}\label{duF}
Let $X = A \sqcup B$. If $A$ and $B$ both have at least 2 points, then
\[ F(X) = F(A) \oplus F(B). \]
\end{thm}
\begin{proof}
First we show that $F(A) \oplus F(B) \subseteq F(X)$. Take some $k\in F(A) \oplus F(B)$, say $k=m+n$ with $m\in F(A)$ and $n\in F(B)$. That means there are two self-maps $f:A\to A$ and $g:B\to B$ with $\#\Fix(f) = m$ and $\#\Fix(g) = n$. Let $h:X\to X$ be defined by:
\[ h(x) = \begin{cases} f(x) \text{ if $x\in A$} \\
g(x) \text{ if $x \in B$} \end{cases} \]
Then 
\[ \#\Fix(h) = \#\Fix(f) + \#\Fix(g) = m + n = k \] 
and so $k \in F(X)$ as desired.

Next we show $F(X) \subseteq F(A) \oplus F(B)$. Take some $k\in F(X)$, so there is some self-map $f$ with $\#\Fix(f) = k$. Let $f_A:A \to X$ and $f_B:B\to X$ be the restrictions of $f$ to $A$ and $B$. 
Since $X=A\cup B$, we have 
\[ \Fix(f) = \Fix(f_A) \cup \Fix(f_B), \]
and $\Fix(f_A) = \Fix(f) \cap A$ and $\Fix(f_B) = \Fix(f)\cap B$. 
Since $A$ and $B$ are disjoint, the union of the fixed point sets above is disjoint. Thus we have $k = \#\Fix(f_A) + \#\Fix(f_B)$.

Since continuous functions preserve connectedness, we must have $f_A(A) \subseteq A$ or $f_A(A) \subseteq B$. Similarly $f_B(B)\subseteq A$ or $f_B(B)\subseteq B$. We show that $k\in F(A)\oplus F(B)$ in several cases.

In the case where $f_A(A) \subseteq B$ and $f_B(B) \subseteq A$, there are no fixed points of $f_A$ or $f_B$, and thus no fixed points of $f$. Thus $k=0$, and it is true that $k\in F(A)\oplus F(B)$ since $0 \in F(A)$ and $0\in F(B)$ by Theorem \ref{0123}.

In the case where $f_A(A)\subseteq B$ and $f_B(B)\subseteq B$, there are no fixed points of $f_A$, and thus $\Fix(f) =\Fix(f_B)$. In this case in fact $f_B$ is a self-map of $B$, and so 
\[ k = \#\Fix(f) = 0 + \#\Fix(f_B) \in F(A) \oplus F(B) \] 
since $0 \in F(A)$ by Theorem \ref{0123} and $\#\Fix(f_B) \in F(B)$ since $f_B$ is a self-map on $B$. The case where $f_A(A)\subseteq A$ and $f_B(B)\subseteq A$ is similar.

The final case is when $f_A(A) \subseteq A$ and $f_B(B) \subseteq B$. In this case $f_A$ is a self-map of $A$ and $f_B$ is a self-map of $B$. Since $\Fix(f) = \Fix(f_A) \cup \Fix(f_B)$, the $k$ fixed points of $f$ must partition into $m$ fixed points of $f_A$ and $n$ fixed points of $f_B$, where $m+n=k$. Thus $m\in F(A)$ and $n\in F(B)$, and so $k=m+n \in F(A)\oplus F(B)$.
\end{proof}

The assumption above that $A$ and $B$ have at least 2 points is necessary. For example if $A$ and $B$ are each a single point, then $F(X) = \{0,1,2\}$ while $F(A)=F(B)=\{1\}$ and thus $F(A)\oplus F(B) = \{2\}$. 

Since any digital image is a disjoint union of its connected components, we have:
\begin{cor}
Let $X_1, \dots, X_k$ be the connected components of a digital image $X$, and assume that $\#X_i > 1$ for all $i$. Then we have:
\[ F(X) = \bigoplus_{i=1}^k F(X_i) \]
\end{cor}

\section{Locations of fixed points}
In many cases, the existence of two fixed points will imply that other fixed points must exist in certain locations. In some cases we will show that $\Fix(f)$ must be connected. We do not
have $\Fix(f)$ connected in general,
as shown by the following.

\begin{exl}
Let $X=\{p_0=(0,0), p_1=(1,0), p_2=(2,0), p_3=(1,1)\}$. 
Let $f: X \to X$ be defined by
\[f(p_0)=p_0,~~~f(p_1)=p_3,~~~f(p_2)=p_2,~~~f(p_3)=p_1.
\]
Then $X$ is $c_2$-connected,
$f \in C(X,c_2)$, and 
$\Fix(f) = \{p_0,p_2\}$ is 
$c_2$-disconnected.
\end{exl}

\begin{lem}
\label{intermediate}
Let $(X,\kappa)$ be a digital image and $f:X\to X$ be continuous. Suppose that $x,x'\in \Fix(f)$ and that $y\in X$ lies on every path of minimal length between $x$ and $x'$. Then $y\in \Fix(f)$.
\end{lem}
\begin{proof}
Let $k$ be the minimal length of a path from $x$ to $x'$. First we show that $y$ must occur at the same location along any minimal path from $x$ to $x'$. That is, we show that there is some $i \in [0,k]_\Z$ with $p(i)=y$ for every minimal path $p$ from $x$ to $x'$. This we prove by contradiction: assume we have two minimal paths $p$ and $q$ with $p(i)=y=q(j)$ for some $j<i$. Then construct a new path $r$ by traveling from $x$ to $y$ along $q$, and then from $y$ to $x'$ along $p$. Then this path $r$ has length less than the length of $p$, contradicting the minimality of $p$.

Thus we have some $i \in [0,k]_\Z$ with $p(i)=y$ for every minimal path $p$ from $x$ to $x'$. Let $p$ be some minimal path from $x$ to $x'$, and since the endpoints of $p$ are fixed, the path $f(p)$ is also a path from $x$ to $x'$. Furthermore the length of $f(p)$ must be at most $k$, and thus must equal $k$ since this is the minimal possible length of a path from $x$ to $x'$. 

Since both $p$ and $f(p)$ are minimal paths from $x$ to $x'$, we have $p(i) = f(p(i)) = y$, and thus $y=f(y)$ as desired.
\end{proof}

A vertex $v$ of a connected graph
$(X,\kappa)$ is an
{\em articulation point} of
$X$ if $(X \setminus \{v\},\kappa)$ is disconnected.
We have the following immediate
consequences of Lemma~\ref{intermediate}.

\begin{cor}
Let $(X,\kappa)$ be a connected
digital image. Let $v$ be 
an articulation point of $X$.
Suppose $f \in C(X,\kappa)$
has fixed points in distinct
components of $X \setminus \{v\}$. Then $v$ is a fixed
point of $f$.
\end{cor}

\begin{cor}
\label{uniqueShortestLem}
Let $(X,\kappa)$ be a digital image and
$f \in C(X,\kappa)$. Suppose
$x,x' \in \Fix(f)$ are such that
there is a unique shortest
$\kappa$-path $P$ in~$X$ from $x$ 
to $x'$. Then $P \subseteq \Fix(f)$.
\end{cor}

\begin{proof}
This follows immediately from
Lemma~\ref{intermediate}.
\end{proof}

\begin{cor}
Let $(X,\kappa)$ be a digital
image that is a tree. Then
$f \in C(X,\kappa)$ implies
$\Fix(f)$ is $\kappa$-connected.
\end{cor}

\begin{proof}
This follows from Corollary~\ref{uniqueShortestLem},
since given $x,x'$ in a tree $X$,
there is a unique shortest path
in~$X$ from $x$ to $x'$.
\end{proof}

For a digital cycle, the fixed point set is typically connected. The only exception is in a very particular case, as we see below.

\begin{thm}\label{cycleconnected}
Let $f:C_n\to C_n$ be any continuous map. 
Then $\Fix(f)$ is connected, or 
is a set of 2 nonadjacent points. The latter case occurs only when $n$ is even and the two fixed points are at opposite positions in the cycle.
\end{thm}
\begin{proof}

If $\#\Fix(f) \in \{0,1\}$, then $\Fix(f)$ is connected. 
When $\#\Fix(f)>1$, we show that
if $x_i,x_j\in \Fix(f)$ are two distinct fixed points, then 
either there is a path from $x_i$ to $x_j$ through other fixed points, or that no other points are fixed.

There are two canonical paths $p$ and $q$ from $x_i$ to $x_j$: the two injective paths going in either 
``direction'' around the cycle. Without loss of generality assume $|p| \ge |q|$. This means that $|q|$ is the shortest possible length of a path from $x_i$ to $x_j$.

Consider the case in which 
$|p| > |q|$. In this case $|q|$ is the unique shortest path from $x_i$ to $x_j$, and by Lemma~\ref{uniqueShortestLem}, $q \subseteq \Fix(f)$, and so $x_i$ and $x_j$ are connected by a path of fixed points as desired.

Now consider the case in which $|p|=|q|$. In this case again $|q|$ is the shortest possible length of a path from $x_i$ to $x_j$, and $p$ and $q$ are the only two paths from $x_i$ to $x_j$ having this length. 
Then $f(q)$ is a path from $x_i$ to $x_j$ of length $|q|$, and so we must have either $f(q)=q$ or $f(q)=p$. 
In the former case, 
$q$ is a path of fixed points connecting $x_i$ and $x_j$ as desired. 
In the latter case,
$\Fix(f) \cap q = \{x_i,x_j\}$.

Similarly considering the path $f(p)$, we must have either $f(p)=p$ (in which case $p$ is a path of fixed points connecting $x_i$ and $x_j$); or $f(p)=f(q)$, in 
which case $\Fix(f) \cap p = \{x_i,x_j\}$. 

Considering all cases,
either a minimal-length path
from $x_i$ to $x_j$ is 
contained in $\Fix(f)$, or
$\Fix(f) = \{x_i,x_j\}$.

The second sentence of the theorem follows from our analysis of the various cases. The only case which gives 2 nonadjacent fixed points requires $x_i$ and $x_j$ to be opposite points on the cycle, which requires $n$ to be even.
\end{proof}

\section{Remarks and examples}
In classical topology $M(f)$ is the only interesting homotopy invariant count of 
the number of fixed points. $S(f)$ is not studied in classical topology, since in 
all typical cases (all continuous maps on polyhedra) we would have $S(f) = [M(f),\infty)_\Z$. 

In classical topology the value of $M(f)$ is generally hard to compute. The 
Lefschetz number gives a very rough indication of homotopy invariant fixed point information, and the more sophisticated Nielsen number is a homotopy invariant lower bound for $M(f)$. See \cite{jiang}.



When $X$ is contractible, all self-maps are homotopic, so $S(f)=F(X)$ for any self-map $f$. 
It is natural to suspect that when $X$ is contractible with $\#X > 1$, we will always have $F(X) = \{0,1,\dots,\#X\}$. This is false, however, as the following example shows:
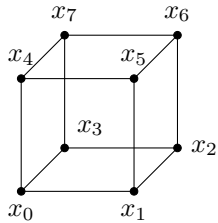
\begin{figure}
\[
\begin{tikzpicture}[scale=1.5] 
\node at (0,0,0) [vertex, label=above right:{$x_3$}] {};
\node at (1,0,0) [vertex, label=right:{$x_2$}] {};
\node at (0,1,0) [vertex, label={$x_7$}] {};
\node at (1,1,0) [vertex, label={$x_6$}] {};
\node at (0,0,1) [vertex, label=below:{$x_0$}] {};
\node at (1,0,1) [vertex, label=below:{$x_1$}] {};
\node at (0,1,1) [vertex, label={$x_4$}] {};
\node at (1,1,1) [vertex, label={$x_5$}] {};
\draw[] (0,0,0) rectangle (1,1,0);
\draw[] (0,0,1) rectangle (1,1,1);
\draw[] (0,0,0) -- (0,0,1);
\draw[] (1,0,0) -- (1,0,1);
\draw[] (0,1,0) -- (0,1,1);
\draw[] (1,1,0) -- (1,1,1);
\end{tikzpicture}
\]
\caption{A contractible image for which $F(X) \neq \{0,1,\dots, \#X\}$.\label{unitcubefig}}
\end{figure}

\begin{exl}
Let $X \subset \Z^3$ be the unit cube of 8 points with $c_1$ adjacency, shown in Figure \ref{unitcubefig}. 
Then $X$ is contractible, so $S(f) = F(X)$ for any 
self-map $f$. By projecting the cube into one of its 
faces, we see that $X$ retracts to $C_4$, and 
since $F(C_4) = \{0,1,2,3,4\}$, we have 
$\{0,1,2,3,4\} \subseteq F(X)$ by Theorem \ref{retractF}.

In fact there are also continuous maps having 5 or 6 fixed points: Let:
\[
g(x_5) = x_0, \quad g(x_6) = x_3, \quad g(x_i) = x_i \text{ for }i\not \in \{5,6\} 
\]
Then $g$ is continuous with $6$ fixed points. Let:
\[ h(x_5) = h(x_7) = x_0, \quad h(x_6) = x_3, \quad h(x_i) = x_i \text{ for }i \not \in \{5,6,7\}\] 
Then $h$ is continuous with $5$ fixed points. Since of course the identity map has 8 fixed points, we have so far shown that $\{0,1,2,3,4,5,6,8\} \subseteq F(X)$.

In fact $7\not \in F(X)$. This follows from
Lemma~\ref{1off}.
We have shown that:
\[ F(X) = \{0,1,2,3,4,5,6,8\}. \]
\end{exl}



The computation of $S(f)$ in general seems to be a difficult and interesting problem. Even in the case of self-maps on the cycle $C_n$, the results are interesting.
First we show that in fact there are exactly 3 homotopy classes of self-maps 
on $C_n$: the identity map $\id(x_i) = x_i$, the 
constant map $c(x_i) = x_0$, and the flip map 
$l(x_i) = x_{-i}$. 

\begin{thm}
\label{3htpyTypes}
Given $f \in C(C_n)$, $f$ is
homotopic to one of: a constant map,
the identity map, or the flip map.
\end{thm}

\begin{proof}
We have noted above that if
$n \le 4$ then $C_n$ is contractible, so every
$f \in C(C_n)$ is homotopic to a
constant map. Thus, in the 
following, we assume $n > 4$.

We can compose $f$ by some rotation $r$ to obtain $g = r\circ f \simeq f$ such that $g(x_0)=x_0$. We will show that $g$ is either the identity, the flip map, or homotopic to a constant map.

If $g$ is not a surjection,
then its continuity implies 
$g(C_n)$ is a connected proper
subset of $C_n$, hence is
contractible. Therefore, $g$ is
homotopic to a constant map.

If $g$ is a surjection, then 
$g$ is a bijection because the domain and codomain of $g$ both have cardinality $n$. By continuity,
$g(x_1) \adj g(x_0) = x_0$. Therefore, either
$g(x_1)=x_{n-1}$ or 
$g(x_1)=x_{1}$.

If $g(x_1)=x_{-1}$, then
          continuity and the fact
          that $g$ is a bijection
          yield an easy induction
          showing that
          $g(x_i)=x_{-i}$,
          $0 \le i < n$. Therefore,
          $g$ is the flip map.

If  $g(x_1)=x_{1}$, a similar argument shows that $g$ is the identity.
\end{proof}

In fact the proof of Theorem \ref{3htpyTypes} demonstrates the following stronger statement. Let $r_d:C_n\to C_n$ be the rotation map $r_d(x_i) = x_{i+d}$. The following generalizes Theorem~3.4 of \cite{Bx10}, which states that any map homotopic to the identity must be a rotation. 
\begin{thm}\label{3maps}
Let $f:C_n \to C_n$ be continuous. Then one of the following is true:
\begin{itemize}
\item $f$ is homotopic to a constant map
\item $f$ is homotopic to the identity, and $f = r_d$ for some $d$
\item $f$ is homotopic to the flip map $l$, and $f = r_d \circ l$ for some $d$
\end{itemize}
\end{thm}

The proof of Theorem \ref{3htpyTypes} also demonstrated that all non-isomorphisms on $C_n$ must be nullhomotopic. Thus all isomorphisms on $C_n$ fall into the second two categories above, and in fact all maps in those two categories are isomorphisms. Thus we obtain:
\begin{cor}
Let $n>4$, and $f:C_n\to C_n$ be an isomorphism with $f\simeq g$ for some $g$. Then $g$ is an isomorphism.
\end{cor}

Now we are ready to compute the values of $S(f)$ for our three classes of self-maps on $C_n$.
\begin{thm}\label{Scycle}
We have $S(f)=\{1\}$ for every $f:C_1\to C_1$.

When $1<n\le 4$, we have $S(f)=\{0,\dots,n\}$ for any $f:C_n \to C_n$.

When $n>4$, let $c$ be any constant map, $\id$ be the identity map, and $l$ be the flip map on $C_n$. We have:
\begin{align*}
S(\id) &= \{0,n\} \\
S(c) &= \{0,1,\dots,\lfloor n/2 \rfloor + 1\} \\
S(l) &= \begin{cases} \{0,1\} &\text{ if $n$ is odd} \\
\{0,2\} &\text{ if $n$ is even} \end{cases}
\end{align*}
\end{thm}
\begin{proof}
When $n=1$, our image is a single point, and the constant map (which is also the identity map) is the only continuous self-map. Thus $S(f) =F(X)= \{1\}$ for every $f:C_1\to C_1$.

When $1<n\le 4$, again all maps are homotopic, and we have $S(f) = F(X)=\{0,\dots,n\}$ for any $f$ by Theorem \ref{0123X}. 

Now we consider $C_n$ with $n>4$, which is the interesting case. 

By Theorem \ref{3maps} the only maps homotopic to $\id$ are rotation maps $r_d$. Since $\#\Fix(r_0) = n$ and $\#\Fix(r_d)=0$ for $d\neq 0$, we have 
\[ S(\id) = \{0,n\}. \]

Now we consider the constant map $c(x_i)=x_0$. 
Let $f \in C(C_n)$ be 
defined as follows.
\[ f(x_i) = \left \{ \begin{array}{ll}
    x_i & \mbox{for } 0 \le i \le \lfloor n/2 \rfloor;  \\
    x_{-i}  &  \mbox{for } \lfloor n/2 \rfloor < i < n.
\end{array}
\right .
\]

\begin{figure}
\[
\begin{tikzpicture}
\foreach \t in {0,...,4} {
 \draw [densely dotted] ({72*\t+90}:1) -- ({72*\t+72+90}:1);
}
\foreach \t in {0,1,4} {
 \node at ({\t*72+90}:1) [vertex, label=$x_\t$] {};
}
\foreach \t in {2,3} {
 \node at ({\t*72+90}:1) [vertex, label=below:$x_\t$] {};
}
\foreach \t in {3,4} {
 \draw [very thick] ({72*\t+90}:1) -- ({72*\t+72+90}:1);
}
\newcommand{\pad}{.1}
\draw [->] (72*1+90:1-\pad) -- (72*4+90:1-\pad);
\draw [->] (72*2+90:1-\pad) -- (72*3+90:1-\pad);
\end{tikzpicture}
\qquad
\begin{tikzpicture}
\foreach \t in {0,...,5} {
 \draw [densely dotted] ({60*\t+90}:1) -- ({60*\t+60+90}:1);
}
\foreach \t in {0,1,5} {
 \node at ({\t*60+90}:1) [vertex, label=$x_\t$] {};
}
\foreach \t in {2,3,4} {
 \node at ({\t*60+90}:1) [vertex, label=below:$x_\t$] {};
}
\foreach \t in {3,4,5} {
 \draw [very thick] ({60*\t+90}:1) -- ({60*\t+60+90}:1);
}
\newcommand{\pad}{.1}
\draw [->] (60*1+90:1-\pad) -- (60*5+90:1-\pad);
\draw [->] (60*2+90:1-\pad) -- (60*4+90:1-\pad);
\end{tikzpicture}
\]
\caption{The map $f$ from Theorem \ref{Scycle} pictured in the cases $n=5$ and $n=6$. All points are fixed except those with arrows indicating where they map to. The path $f(C_n)$ is in bold.\label{Scyclefig}}
\end{figure}
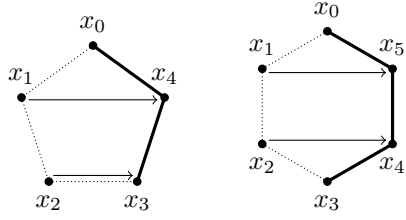

This map ``folds'' the cycle onto a path that is ``about half" of the cycle, with $\lfloor n/2 \rfloor+1$ fixed points. See Figure \ref{Scyclefig}. This 
can be taken as the first
step of a homotopy, in which
successive steps shrink the
path and the number of fixed
points by one per step,
until a constant map is
reached at the end of the
homotopy.
Thus $\{1,\dots,\lfloor n/2 \rfloor +1\} \subseteq S(c)$, and of course $0\in S(c)$ also by Theorem \ref{0123}. 

Thus we have shown there is
a fixed path $p$ between fixed points of $f$,
$x_i,x_j$, of
length at least
$\lfloor n/2 \rfloor + 1$.
We wish to show that in fact $S(c) = \{0,\dots,\lfloor n/2 \rfloor + 1\}$. We show this by contradiction: 
take some nullhomotopic $f$, assume that $k \in S(f)$ with $k>\lfloor n/2 \rfloor + 1$,
and we will show in fact that all points are fixed; this
would be a contradiction 
since $f \neq \id$.
Choose any $x \in C_n \setminus p$. 
Then $x$ lies on the unique shortest path from $x_i$ to $x_j$. Then $x\in \Fix(f)$ by Lemma \ref{uniqueShortestLem}; this gives our desired
contradiction.



Finally we consider the flip map $l(x_i) = x_{-i}$. By Theorem \ref{3maps}, all maps homotopic to $l$ have the form $f(x_i) = r_d \circ l (x_i) = x_{d-i}$. Such a map has a fixed point at $x_i$ if and only if $d=2i \pmod n$. When $d$ is odd there are no solutions, and so $\#\Fix(f) = 0$. When $d$ is even, say $d=2a$, and $n$ is odd, there is one solution $i=a$. When $d$ is even and $n$ is also even, there are two solutions: $i=a$ and $i=a+n/2$. Thus we have some maps with no fixed points, and when $k$ is odd we have some with one fixed point, and when $k$ is even we have some with two. We conclude:
\[ S(l) = \begin{cases} \{0,1\} &\text{ if $n$ is odd} \\
\{0,2\} &\text{ if $n$ is even} \end{cases} \qedhere \]
\end{proof}

By Theorem \ref{3htpyTypes}, any self-map on $C_n$ is homotopic to the constant, identity, or flip. Thus by taking unions of the sets above, we have:
\begin{cor}
\[ F(C_n) = \begin{cases} \{1\} &\text{ if } n=1, \\
\{0,\dots,n\} &\text{ if } 1<n\le 4, \\
\{0,1,\dots,\lfloor n/2 \rfloor + 1, n\} &\text{ if } n>4.\end{cases}
\]
\end{cor}

From the Corollary above we see that $F(C_5)=\{0,1,2,3,5\}$ and $F(C_6) = \{0,1,2,3,6\}$, and thus the formula of Corollary \ref{0123X} is exact in these two cases. The images $C_5$ and $C_6$ are the only examples known to the authors in which this occurs. 

\begin{question}
Is there any digital image $X \not\in\{C_5,C_6\}$ with $\#X > 4$ and $F(X) = \{0,1,2,3,\#X\}$?
\end{question}

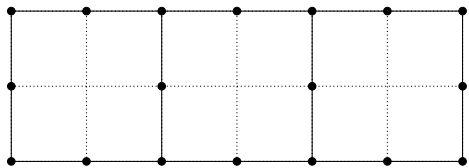
\begin{figure}
\[
\begin{tikzpicture}
\draw[densely dotted] (0,0) grid (6,2);
\foreach \x in {0,...,6} {
 \node at (\x,0) [vertex] {};
 \node at (\x,2) [vertex] {};
 }
\foreach \x in {0,2,4,6}
 \node at (\x,1) [vertex] {};
\draw[step=2] (0,0) grid (6,2);
\end{tikzpicture}
\]
\caption{An image having a self-map with $X(f)=0$\label{xexample}}
\end{figure}

We conclude this section with two interesting examples showing the wide variety of fixed point sets that can be exhibited for other digital images. Tools we use in our discussion 
include the following.

A path
$r: [0,m]_{\Z} \to X$
that is an isomorphism onto $r([0,m]_{\Z})$
is a {\em simple path}.
If a loop
$p$ is an isomorphism onto $p(C_m)$,
$p$ is a {\em simple loop}.

\begin{definition}
{\rm \cite{hmps}}
A simple path or a simple loop in
a digital image $X$ {\em has no
right angles} if no pair of 
consecutive edges of the path or
loop belong to a loop of length~4
in $X$.
\end{definition}

\begin{definition}
{\rm \cite{hmps}}
A {\em lasso in $X$} is a simple loop
$p: C_m \to X$ and a path
$r: [0,k]_{\Z} \to X$ such that $k>0$,
$m \ge 5$, $r(k)=p(x_0)$, and
neither $p(x_1)$ nor $p(x_{m-1})$
is adjacent to $r(k-1)$.

The lasso {\em has no right angles}
if neither $p$ nor $r$ has a right
angle, and no right angle is formed
where $r$ meets $p$, i.e., the
final edge of $r$ and neither of
the edges of $p$ at $p(x_0)$ form
2 edges of a loop of length 4
in~$X$.
\end{definition}

\begin{thm}
{\rm \cite{hmps}}
\label{lassoThm}
Let $X$ be an image in which, 
for any two adjacent points 
$x \adj x' \in X$, there is a 
lasso with no right angles having 
path $r: [0, k]_{\Z} \to X$
with $r(0) = x$ and
$r(1) = x'$. Then $X$ is rigid.
\end{thm}

\begin{exl}
\label{has1MiddleBar}
Let $X$ be the digital image
\[ X = \left ([0,6]_{\Z} \times \{0,2\} \right )
\cup \{(0,1), (2,1), (4,1), (6,1)\}
\]
(see Figure~\ref{xexample}), with 4-adjacency.

By Theorem~\ref{lassoThm},
this image is rigid. 
It is easy, though a bit tedious, to verify that the hypothesis of Theorem \ref{lassoThm} is satisfied by $X$. For example, in Figure \ref{lassofig} we exhibit a lasso with no right angles for two adjacent points. It is easy to construct such lassos for any pair of adjacent points.

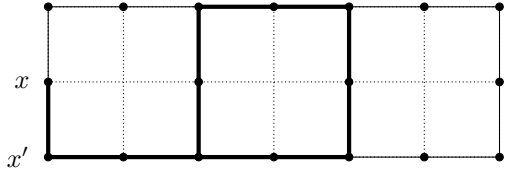
\begin{figure}
\[
\begin{tikzpicture}
\draw[densely dotted] (0,0) grid (6,2);
\foreach \x in {0,...,6} {
 \node at (\x,0) [vertex] {};
 \node at (\x,2) [vertex] {};
 }
\node at (0,1) [label=left:$x$] {};
\node at (0,0) [label=left:$x'$] {};
\foreach \x in {0,2,4,6}
 \node at (\x,1) [vertex] {};
\draw[step=2] (0,0) grid (6,2);
\draw [ultra thick] (0,1) -- (0,0);
\foreach \x in {0,...,3} {
 \draw [ultra thick] (\x,0) -- (\x+1,0);
}
\foreach \x in {2,3} {
 \draw [ultra thick] (\x,2) -- (\x+1,2);
}
\foreach \y in {0,1} {
 \draw [ultra thick] (2,\y) -- (2,\y+1);
 \draw [ultra thick] (4,\y) -- (4,\y+1);
}
\end{tikzpicture}
\]
\caption{A lasso for the points $x=(0,1)$ and $x'=(0,0)$\label{lassofig}}
\end{figure}

Since $X$ is rigid, we have $S(\id) = \{\#X\} = \{18\}$. Let $f:X\to X$ be the 180-degree rotation of $X$. Then $f$ is an isomorphism, and so by
Theorem~\ref{isoOfRigidIsRigid}, $f=f \circ \id_X$ is rigid.
Thus $S(f) = \{\#\Fix(f)\} = \{0\}$. 
In particular this provides an example for the question posed 
in~\cite{bs19} if $X(f)$ could ever equal 0 for a connected image.
\end{exl}

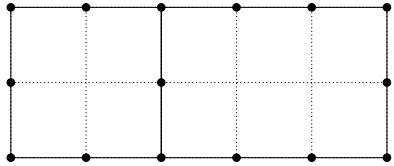
\begin{figure}
\[
\begin{tikzpicture}
\draw[densely dotted] (0,0) grid (5,2);
\foreach \x in {0,...,5} {
 \node at (\x,0) [vertex] {};
 \node at (\x,2) [vertex] {};
 }
\foreach \x in {0,2,5}
 \node at (\x,1) [vertex] {};
\draw (0,0) rectangle (2,2);
\draw (2,0) rectangle (5,2);
\end{tikzpicture}
\]
\caption{An image with many different values for $S(f)$\label{sexample}}
\end{figure}

The following example demonstrates an image which has many different possible sets which can occur as $S(f)$ for various self-maps $f$. 
\begin{exl}
\label{self-mapExamples}
Let
\[X = \left ([0,5]_{\Z} \times \{0,2\} \right )
    \cup \{(0,1), (2,1), (5,1)\}
    \]
(see Figure \ref{sexample}), with 4-adjacency. In this image we have several different homotopy
classes of maps. We will derive sufficient
information about $S$ for some of these to compute $F(X)$.

By Theorem \ref{lassoThm}, $X$ is rigid, so $S(\id) = \{\#X\} = \{15\}$. 

Let $f$ be a vertical reflection. Then $f$ is rigid by Corollary \ref{rigidiso}, and has 3 fixed points, so $S(f) = \{3\}$. 

Let $g$ be the function that
maps the bottom horizontal bar onto the top one, and fixes all other points. 
Then $g$ has 9 fixed points, and is homotopic to a constant map. 
We can retract the image of $g$ down to a point one point at a time, and so $\{0,1,\dots, 9\} \subseteq S(g)$.

Let $h$ be the function which maps the left vertical bar into the middle vertical bar and fixes all other points. 
Then $h$ has 12 fixed points. We can additionally map one or both of the next two points into the middle 
vertical bar to obtain maps homotopic to $h$ with 11 or 10 fixed points. We can do 
these retractions followed by a rotation around the 10-cycle on the right to 
obtain a map homotopic to $h$ with no fixed points. 
Thus $\{0,10,11,12\} \subseteq S(h)$. 

We have $F(X) \cap \{13,14\} = \emptyset$ by
Proposition~\ref{pullThm}, since for all
$x \in X$ we see easily that $P(x) \ge 3$.

We therefore have
\[ F(X) = \{0,1,2,3,4,5,6,7,8,9,10,11,12,15\}. \]

%
%
%
\end{exl}

\section{Further remarks}
We have introduced and
studied several measures
concerning the fixed point
set of a continuous self-map
on a digital image. We anticipate further research
in this area.

https://www.sciencedirect.com/science/article/abs/pii/0167865594900124

\bibitem{Bx99}
 L. Boxer, A Classical Construction for the Digital Fundamental Group, Journal of Mathematical Imaging and Vision 10 (1999), 51-62.
 
 https://link.springer.com/article/10.1023/A

\bibitem{Bx10}
L. Boxer, Continuous maps on digital simple 
closed curves,
{\em Applied Mathematics} 1 (2010), 377-386.

https://www.scirp.org/Journal/PaperInformation.aspx?PaperID=3269

 \bibitem{BxNormal}
 L. Boxer, Generalized normal product adjacency in digital topology, 
 {\em Applied General Topology} 18 (2) (2017), 401-427
 
 https://polipapers.upv.es/index.php/AGT/article/view/7798/8718

\bibitem{BxAlt}
L. Boxer, Alternate product adjacencies in digital topology,
{\em Applied General Topology} 19 (1) (2018), 21-53

https://polipapers.upv.es/index.php/AGT/article/view/7146/9777

\bibitem{BEKLL}
 L. Boxer, O. Ege, I. Karaca, J. Lopez, and J. Louwsma, Digital fixed points, approximate fixed points, and universal functions, 
 {\em Applied General Topology}
 17(2), 2016, 159-172.
 
 https://polipapers.upv.es/index.php/AGT/article/view/4704/6675
 
 \bibitem{BK12}
 L. Boxer and I. Karaca,
 Fundamental groups for digital products, 
 {\em Advances and Applications in Mathematical Sciences} 11(4) (2012), 161-180.
 
 http://purple.niagara.edu/boxer/res/papers/12aams.pdf
 

\bibitem{bs19} L. Boxer and P.C. Staecker, Remarks on fixed point assertions in digital topology, \emph{Applied General Topology}, to appear.

https://arxiv.org/abs/1806.06110


\bibitem{hmps} 
J. Haarmann, M.P. Murphy, C.S. Peters, and P.C. Staecker,
Homotopy equivalence in finite digital images,
{\em Journal of Mathematical Imaging and Vision}
53 (2015), 288-302.

https://link.springer.com/article/10.1007/s10851-015-0578-8


\bibitem{jiang} 
B. Jiang, Lectures on Nielsen fixed point theory,
{\em Contemporary Mathematics} 18 (1983).

\bibitem{Khalimsky}
E. Khalimsky,
 Motion, deformation, and homotopy in finite spaces, in
{\em Proceedings IEEE Intl. Conf. on Systems, Man, and Cybernetics},
1987, 227-234.


\bibitem{Rosenfeld}
A. Rosenfeld, `Continuous' functions on digital pictures, {\em Pattern Recognition Letters} 4,
pp. 177-184, 1986.

Available at https://www.sciencedirect.com/science/article/pii/0167865586900176


\bibitem{stae15}
P.C. Staecker, Some enumerations of binary digital images, arXiv:1502.06236, 2015

\end{document}